\documentclass[12pt,bezier]{article}

\usepackage{amsmath, amssymb, amsfonts, amsthm,url}
\usepackage{xcolor,colortbl}
\usepackage{graphics}
\usepackage{array}
\usepackage{makecell}
\usepackage{tikz}
\usepackage{tkz-graph}
\usepackage{tkz-berge}
\usepackage{color}
\usepackage[utf8]{inputenc}
\usepackage{graphicx}
\usepackage{hyperref}
\usepackage{amssymb,amsmath,amsthm,latexsym,tikz,url,float,subfig,caption,graphicx,pgfplots,mathrsfs,hyperref}
 \usetikzlibrary{matrix}
\tikzstyle{vertex}=[circle, draw, inner sep=0pt, minimum size=3pt]

\usepackage{color}
\newtheorem{theorem}{Theorem}[section]
\newtheorem{lemma}[theorem]{Lemma}
\theoremstyle{definition}

\newtheorem{example}[theorem]{Example}

\newtheorem{remark}[theorem]{Remark}

\newcommand{\Tau}{\mathcal{T}}

\newcommand{\bz}{\mathbb{Z}}
\newcommand{\CS}{\mathrm{CS}}
\begin{document}
\title{Total perfect codes in Cayley sum graphs of cyclic groups}
\author{ Masoumeh Koohestani$^{\,\rm a}$
\quad Doost Ali Mojdeh  $^{\,\rm a,}$\thanks{Corresponding author: \tt dmojdeh@umz.ac.ir}  \quad
 Mohsen Ghasemi $^{\,\rm b}$ 
\\[.3cm]
{\sl\normalsize $^{\rm a}$ Department of Mathematics, Faculty of Mathematical Sciences, }\\
{\sl\normalsize University of Mazandaran, Babolsar, Iran}\\
{\sl\normalsize $^{\rm b}$Department of Mathematics, Urmia University,}\\
{\sl\normalsize  P. O. Box 575615-1818, Urmia, Iran}
}
\maketitle
\footnote{{\em E-mail Addresses}: {\tt  m.koohestani@umail.umz.ac.ir,
m.ghasemi@urmia.ac.ir }}

\maketitle
\begin{abstract} 
In this paper, we consider Cayley sum graphs over the cyclic group $\mathbb{Z}_n$
and aim to explore several necessary and sufficient conditions for the existence of total perfect codes in these graphs. Specifically, we examine various cases for the connection set of the graph, including when it is periodic, aperiodic, or square-free. To this end, we utilize a correspondence that we first establish between total perfect codes and factorizations of groups, along with their algebraic properties. We then generalize some of these conditions to the direct product of cyclic groups, i.e. 
\texorpdfstring{$\bz _{n_1} \times\dots \times \bz_{n_d}$.}

\vspace{5mm}
\noindent {\bf AMS Mathematics Subject Classification\,(2020):}   05C25, 05C69, 94B99
\end{abstract}

\section{Introduction}
Since the advent of coding theory in the late 1940s, perfect codes have been fundamental to the development of the information theory \cite{Heden-2008, Lint-1975}. The importance of perfect codes is well recognized, with Hamming and Golay codes among noteworthy examples. The notion of perfect codes can be naturally extended to the context of graphs \cite{Biggs-1973}. In graph theory, perfect codes are also referred to as efficient dominating sets \cite{Dejter-2003,Knor-2012} or independent perfect dominating sets \cite{Lee-2001}. In the specific context of Cayley sum graphs, the study of perfect codes and total perfect codes has recently attracted substantial attention \cite{Zhang-2024, Wang-2024}.

Let $\Gamma$  be a graph with vertex set $V (\Gamma)$ and edge set $E(\Gamma)$.  A \textit{perfect code}
 in $\Gamma$ is a subset $C$ of $V(\Gamma)$ such that no two vertices in $C$ are adjacent and every vertex in $V(\Gamma)\setminus C$ is adjacent to exactly one vertex in $C$. Also, $C$ is called a \textit{total perfect code} in  $\Gamma$,  if every vertex in  $V (\Gamma)$ has exactly one neighbor in $C$. 
Let $G$ be a group under addition with the identity element 0. An element $x$ of $G$ is called a \textit{square} if $x = y+y$, for some element
$y \in G$. A subset of $G$ is called a  \textit{square-free subset} of $G$ if it is a set without squares. A
subset $S$ of $G$ is called a  \textit{normal subset} if $S$ is a union of some conjugacy classes of $G$ or
equivalently, for every $g \in G$, $-g+S+g := \{-g+s+g : s \in S\} = 
S$. Remark that any subset of an
abelian group is normal. Let $S$ be a a normal subset of $G$. The  \textit{Cayley sum graph} ${\rm CS}(G, S)$ of $G$ with respect to the
 \textit{connection set} $S$ is a graph whose vertex set is $G$ and  two vertices $g$ and $h$ being
adjacent if and only if $g+h \in S$ and $g \neq h$. 
A perfect code (total perfect code) of  ${\rm CS}(G, S)$ which is also a subgroup of $G$ is called
a  \textit{subgroup perfect code} (\textit{subgroup total perfect code}) of $G$. A subgroup perfect code (subgroup total perfect code) $H$  of the group $G$ is said to be \textit{non-trivial} when $H \neq \{1\}$.  
Additionally, it is obvious that the neighborhood
of a vertex $g$ is $S-g$ if $2g \notin S$ and $(S \setminus \{2g\})-g$ if $2g \in S$. 
Therefore ${\rm CS}(G, S)$ is
a $|S|$-regular graph if and only if either $S$ is square-free in $G$ or $S$ consists of all squares
of $G$. 
In what follows, we summarize the main  points that we will cover in this paper.

In Section \ref{Prel}, we review some basic results, definitions and notations that will be used to prove our main theorems in the following sections. More importantly, we will see that there exists a one-to-one correspondence between total perfect codes of Cayley sum graphs and factorizations of groups that are defined in \cite{Vuza-1991}.  A \textit{factorization} of a group $G$ is a tuple $(A_1, \dots, A_d)$, where $A_i$s are subsets of $G$ and $d\geqslant 2$, such that each element $g\in G$ can be uniquely written as 
$$g=a_1+\dots +a_d,$$
 with $a_i\in A_i$ for $1\leqslant i\leqslant d$.

 In Section \ref{Cyclic}, our main results are focused on  finding sufficient and necessary conditions for a Cayley sum graph over a cyclic group to admit a total perfect code. 
 More precisely, it is proved in \cite[Lemma 2.5]{Feng-2017} 
 that a connected circulant graph $\mathrm{Cay}(\bz _n, S)$, where $|S|\mid n$, admits a  total perfect code if $s\not\equiv s^{\prime}~ (\mathnormal{mod}\ k)$, for distinct $s, s^{\prime}\in S$. We also prove a similar condition for Cayley sum graphs but this condition is not  always necessary. So, we will discuss some situations in which it is also necessary. This leads to several theorems which are summarized in Table \ref{survey_table}.
 Moreover, in Section \ref{Direct-Sec}, we consider the direct product of $d$ cyclic groups and  generalize some theorems of the previous section.
\section{Preliminaries}\label{Prel}
In this section, we review some results that will be  applied to prove our main results in the following sections.
Prior to this, we introduce some definitions and notations. 

Suppose that $G$ is a group and $X$ and $Y$ are subsets of $G$. We write $X+Y=\{x+y:~ x\in X, y\in Y\}$, $X-Y=\{x-y: ~x\in X, y\in Y\}$, and $G_X=\{g\in G:~ X+g=X\}$. 
Moreover,  we write $X/H=\{H+x: ~ x\in X\}$, where  $H$ is a subgroup of $G$. Each element of $G_X$
is called a \textit{period} of $X$ in $G$.
 Also, $G_X$ is called the \textit{subgroup of periods} or the \textit{stabilizer} of $X$ in $G$.
If $G_X\neq \{0\}$, $X$ is called \textit{periodic}, where $X\neq \emptyset$, and \textit{aperiodic} otherwise \cite{Vuza-1991}.
 Furthermore, when $G$ is factorized into two factors $X$ and $Y$ we denote it by $G=X\oplus Y$.
  If  at least one factor in every factorizations of an abelian group is periodic, then it is said to be a \textit{good abelian} group. It is well known that the cyclic groups whose order belong to the following set 
  \begin{equation}\label{GoodAbeli}
  N=\{p^{\lambda}, p^{\lambda}q, p^2q^2, pqr, p^2qr, pqrs:~   p, q, r, s ~\text{are distinct primes and}~\lambda \geqslant 1
 \}
\end{equation}
are good abelian groups, see \cite[Section 2]{Tijdeman-1995}.
 We use the notation $G_1\cong G_2$ to show that  two groups $G_1$  and $G_2$ are isomorphic. Also, when an integer $x$ divides or  does not divide the integer $y$, we denote them by $x\mid y $ and $x\nmid y$, respectively. Moreover, the notation $(x,y)=1$ means that  the integers $x$ and $y$ are coprime to each other. 
 Finally, but importantly, the set of all total perfect codes of ${\rm CS}(G, S)$ is denoted by $\mathcal{T} (G, S)$.
\begin{lemma}\label{transv-fac}
Let $G$ be a group, $S$ a normal subset of $G$, and $C$ be a subgroup of $G$. 
 Then, $C$ is a total perfect code of $\mathrm{CS}(G,S)$ of degree $|S|$ if and only if $G=C\oplus S$. 
\end{lemma}
\begin{proof}
Let $C$ be a total perfect code of $\mathrm{CS}(G,S)$. So, for any $x\in G$, there exists a unique $c\in C$ such that $x+c=s$, for some $s\in S$.
So, each $x\in G$ can be uniquely written as 
$x=(-c)+s$. This leads to $G=C\oplus S$.

Let $G=C\oplus S$, where $S$ is a normal subset of $G$, and $C$ is a subgroup of $G$. Then,  each $x\in G$ can be uniquely written as $x=c+s$ for $c\in C$ and $s\in S$. So, $C\in \Tau (G,S)$.
\end{proof}
\begin{remark}\label{trans-fac-total}
In Lemma \ref{transv-fac}, we can view $C$ as a subset of $G$. In this case, $C$ is a total perfect code in $\mathrm{CS}(G, S)$ if and only if $G = (-C) \oplus S$.
\end{remark}
\begin{lemma}\label{factor-conditions}\cite[Proposition 2.1]{Vuza-1991}
Suppose that $G$ is an abelian group, and  $X$ and $Y$ are non-empty subsets of $G$. Then $G = X \oplus Y$ if and only if any two of the following conditions hold:
\begin{itemize}
\item[(a)] $G=X+Y$,
\item[(b)]$(X-X)\cap (Y-Y)=\{0\}$,
\item[(c)] $|G|=|X||Y|$.
\end{itemize}
\end{lemma}
\begin{lemma}\label{factor-prime}\cite[Theorem 2.1]{Vuza-1991}
Suppose that $n\geqslant 2$ is an integer and $X$ and $Y$ are non-empty subsets of $\bz _n$ with $|X|=p^r$, where $p$ is a prime and $r\geqslant 0$. If $(X,Y)$ is a factorization of
 $ \mathbb{Z}_n$, then at least one of $X$ and $Y$ is periodic.
\end{lemma}
\begin{lemma}\label{factor-GX}\cite[Lemma 2.5]{Yun-2017}
Suppose that $G$ is an abelian group and $X$ and $Y$ are two non-empty subsets of $G$. We have $G=X\oplus Y$ if and only if $G/G_X=\left(X/G_X\right)\oplus \left(Y/G_X\right)$ and $G_X\cap (Y-Y)=\{0\}$.
\end{lemma}
\begin{lemma}\label{GX-aperiodic}\cite[Lemma 2.7]{Cameron-2025}
Let $G$ be an abelian group and 
$\emptyset \neq X\subseteq G$. Then $X/G_X$ is an aperiodic subset of $G/G_X$.
\end{lemma}
\begin{lemma}\cite[Lemma 2.6]{Cameron-2025}\label{|H|=|X|}
Suppose that $n\geqslant 2$ is an integer and $X\subseteq \bz_n$, where $|X|\geqslant 2$. Let $H$ be the subgroup of periods of $X$ in $\bz _n$. If $|X|=|H|$ and $\left(|X|, n/|X|\right)=1$, then
$x\not\equiv x^{\prime}~ (\mathnormal{mod}\ |X|)$, for distinct $x, x^{\prime}\in X$.
\end{lemma}
Note that the following lemma is derived from \cite[Theorem 2]{Amooshahi-2016}. Although the definition of the Cayley sum graph in \cite{Amooshahi-2016} differs slightly from the one used in the present context, the results remain applicable here, as our Cayley sum graphs are obtained by simply removing all loops from those considered in that work.

\begin{lemma}\cite[Theorem 2]{Amooshahi-2016}\label{Connectivity-CS}
Let $S$ be a normal subset of a group $G$. Then the Cayley sum graph $\rm {CS}(G, S)$ is connected if and only if $G=\left\langle S\right\rangle$ and $|G : \left\langle S-S\right\rangle|\leqslant 2$.
\end{lemma}
\section{Cayley Sum Graphs Over Cyclic Groups }\label{Cyclic}
In \cite{Feng-2017}, a sufficient condition has been proved for a Cayley graph to admit a total perfect code but this condition is not necessary.  In \cite{Cameron-2025}, they detect some situations in which this condition is also necessary. We  now aim to discuss some sufficient 
and necessary conditions for Cayley sum graphs. The following lemma is the counterpart of \cite[Lemma 2.5]{Feng-2017} for Cayley sum graphs.

\begin{table}[ht]
\begin{center}
\begin{tabular}{|c|c|c|c|c|c|}
\hline
$n$ & $S$ & $C$ &$(|S|,|C|)=1$ & \makecell{\text{All total}\\ \text{ perfect codes}\\ \text{are recognized?}} & \text{Reference} \\
\hline
&  & $\leqslant \bz_n$ &\text{Not Required}  & \text{Yes} & Theorem \ref{k|n-distinct-Th} \\
\hline
$=p|S|$ &\makecell{ $|S|=p$\\ \text{Square-free} }&  &\text{Not Required } & \text{Yes} & \text{Theorem \ref{S/H}} \\
\hline
 & $|S|=p>2$ &  &  \text{Not Required}& \text{No} & \text{Theorem \ref{p-circulant}}  \\
\hline
 & $|S|=p^l$ &  & \text{Required} & No & \text{Theorem \ref{pl-circulant}} \\
\hline
 &  \makecell{$|S|=pq$\\ \text{Square-free}\\ \text{Periodic}}&  & \text{Required} &\text{No}  &\text{Theorem \ref{Periodic-pq}}  \\
\hline
 $\in \{pqr, pqrs\}$& \makecell{$|S|=pq$\\ \text{Square-free}\\ \text{Aperiodic}} &  & \text{Required} &\text{No}  &\text{Lemma \ref{Lem-pq}}   \\
\hline
$\in N$ &\text{Square-free}  &  & \text{Required} &  \text{No}& \text{Theorem \ref{n-in-N}}  \\
\hline
\end{tabular}
\end{center}
\caption{Under which conditions  does $CS(\bz _n, S)$, with $|S|\mid n$, admits a total perfect $C$  if and only if $s\not\equiv s^{\prime}~ (\mathnormal{mod}\ |S|)$, for distinct $s, s^{\prime}\in S$? }
\label{survey_table}
\end{table}
\begin{lemma}\label{k|n-distinct}
Suppose that $\Gamma:=\mathrm{CS}(\mathbb{Z}_n,S)$ is  connected of degree $k=|S|$ and order $n\geqslant 4$. If
 $k\mid n $ and $s\not\equiv s^{\prime}~ (\mathnormal{mod}\ k)$, for distinct $s, s^{\prime}\in S$, then $\Gamma$ admits a  total perfect code. Furthermore, the total perfect code is of the form $k\bz_n$ which is a subgroup of $\bz_n$.
\end{lemma}
\begin{proof}
Since the elements of $S$ are pairwise distinct modulo $k$, each element 
 $x\in\mathbb{Z}_n$ can be uniquely written as $x\equiv s~ (\mathnormal{mod}\ k)$, 
for some $s\in S$. So, $\{ki:~ 0\leqslant i< n/k\}=k\bz_n$ is a total perfect code of $\mathrm{CS}(\mathbb{Z}_n,S)$. 
\end{proof}
\begin{theorem}\label{k|n-distinct-Th}
A connected Cayley sum graph $\mathrm{CS}(\mathbb{Z} _n,S)$ of degree $k=|S|$ and order $n\geqslant 4$, where $k\mid n $,  admits a subgroup total perfect code if and only if $s\not\equiv s^{\prime}~ (\mathnormal{mod}\ k)$, for distinct $s,s^{\prime}\in S$.
\end{theorem}
\begin{proof}
The sufficiency is proved by Lemma \ref{k|n-distinct}. 
For the necessity, suppose that $n$ and $S$ are as in the theorem and  $\mathrm{CS}(\mathbb{Z}_n,S)$ admits a subgroup total perfect code $C$. By Lemma \ref{transv-fac}, we have  $\mathbb{Z}_n = S \oplus C$. Hence, from Lemma \ref{factor-conditions},  $(S -S)\cap(C -C) = \{0\}$ and $|C| = n/ k$.  Therefore, $C=k\mathbb{Z}_n$, because it is the only subgroup of $\mathbb{Z}_n$ with order $n/k$.  We have $C-C = k\mathbb{Z}_n$, and so, from the fact that $(S -S)\cap(C -C) = \{0\}$, we must have $s\not\equiv s^{\prime}~ (\mathnormal{mod}\ k)$, for distinct $s,s^{\prime}\in S$.
\end{proof}
\begin{lemma}\label{IsomorphicGroups}
Let $G_1$ and $G_2$ be finite groups such that $G_1 \cong G_2$ via the group isomorphism $\varphi: G_1\rightarrow G_2$. Then for a subset  $C \subseteq G_1$ we have 
$$ C\in \Tau (G_1,S) \Leftrightarrow \varphi (C) \in \Tau\left(G_2,\varphi (S)\right),$$
where $S$ is a normal subset of $G_1$.
\end{lemma}
\begin{proof}
Suppose that $ C\in \Tau (G_1,X)$. For any $g_2\in G_2$, we have $g_2=\varphi(g_1)$, for some $g_1\in G_1$. On the other hand, there exists a unique element $c\in C$ such that $g_1+c=s$, for some $s\in S$. So, $g_2+\varphi(c)=\varphi(s)$, and hence, $\varphi (C) \in \Tau\left(G_2,\varphi (S)\right)$. Note that $\varphi(S)$ is a normal subset of $G_2$ because for any $\varphi(s)\in \varphi (S)$ and $g_2\in G_2$, we have $-g_2+\varphi(s)+g_2=\varphi(-g_1)+\varphi(s)+\varphi(g_1)=\varphi(-g_1+s+g_1)
\in \varphi(S)$, for some $g_1\in G_1$.

The sufficiency can be proved by a similar argument.
\end{proof}
\begin{theorem}\label{S/H}
Let $\Gamma:=\mathrm{CS}(\mathbb{Z}_n,S)$ be connected, where $n\geqslant 4$ and $S$ is a square-free subset of $\mathbb{Z}_n$, with $n=p|S|$ for some prime $p$. Suppose that $H$ is the subgroup of periods of $S$ in $\bz_n$ under addition. Then $\Gamma$ has a total perfect code if and only if $p=2$ and $s\not\equiv s^{\prime}~\left(\mathnormal{mod}\ |S|/|H| \right)$, for distinct $H+s, H+s^{\prime}\in S/H$. In this case, the total  perfect codes of $\Gamma$ have the general form 
$$\{h_1+i, h_2+i+k\},$$ where $h_1,h_2\in H$ and $0\leqslant i\leqslant k$.
\end{theorem}
\begin{proof}
We consider the following two cases separately.

\textsf{Case 1.} $S$ is aperiodic. So, $H=\{0\}$ and the sufficiency is proved by Lemma \ref{k|n-distinct}.

To prove the necessity, let $C\in \Tau(\mathbb{Z}_n, S)$. Then $|C|$ is even, and from Remark \ref{trans-fac-total}, $\bz _n=(-C)\oplus S$. By Lemma \ref{factor-conditions}, we have $|-C|=|C|=n/|S|=p$, and hence, $p=2$. From Lemma \ref{factor-prime}, $(-C)$ must be periodic.
Suppose that $K$ is the subgroup of periods of $-C$ in $\bz_n$ under addition.
 We know that $|K|\mid |C|$. 
 Since $|K|>1$,  we have $|K|=2$, and hence, $K=(n/2)\mathbb{Z}_n$. Since $|K|=|-C|$, $-C$ is a coset of $K$ in $\bz_n$. Therefore, $-C$ and $C$ are of the form $(n/2)\mathbb{Z}_n+i$, for some 
$i\in \mathbb{Z}_n $.  In particular, $(n/2)\mathbb{Z}_n\in \Tau(\bz_n,S)$. So, Lemma \ref{k|n-distinct} implies that  $s\not\equiv s^{\prime}~ (\mathnormal{mod}\ |S|)$, for distinct $s, s^{\prime}\in S$. 

\textsf{Case 2.} $S$ is periodic. So, $|H|>1$. Set $k=|S|/|H|$. So, we have $|\mathbb{Z}_n|/|H|=n/|H|=kp$, and hence, $H=(kp)\mathbb{Z}_n$. We claim that  $H\notin S/H$, or equivalently, $S\cap H=\emptyset$.  Suppose to the contrary and let $s\in S\cap H$.
Since $H$ is the stabilizer of $S$, $S+s=S$. Thus, $2s=s+s\in S$ which is a contradition because $S$ is square-free.
 Set $S/H=\{H+t_1, H+t_2,\cdots,H+t_k\}$, where $0<t_i\leqslant kp-1$ and $1\leqslant i\leqslant k$. From Lemma \ref{transv-fac} and Lemma \ref{factor-GX}, we conclude that
\begin{align}\label{Total-Z-Z/H}
C\in \Tau (\mathbb{Z}_n,S)&\Leftrightarrow \mathbb{Z}_n=C\oplus S\nonumber\\
&\Leftrightarrow \mathbb{Z}_n/H=(C/H)\oplus (S/H), ~ H\cap (C-C)=\{0\}\nonumber\\
& \Leftrightarrow C/H\in\Tau(\mathbb{Z}_n/H,S/H).
\end{align}
Since $H=(kp)\mathbb{Z}_n$, we have $s\not\equiv s^{\prime}~ (\mathnormal{mod}\ k)$, for distinct $H+s, H+ s^{\prime}\in S/H$, if and only if $t_i\not\equiv t_j~ (\mathnormal{mod}\ k)$, for distinct $i, j\in \{1, 2, \cdots, k\}$.

Now, note that the set 
\begin{equation}\label{SetT}
T=\{t_1, t_2, \dots, t_k\},
\end{equation}
 is aperiodic in $\mathbb{Z}_n$. Suppose, for contradiction, that it is not. Then there exists some $g \in \mathbb{Z}_n$, with $g \ne 0$, such that $T + g = T$. This means that for every $t_i$ in the set, there exists some $t_j$ such that $t_i + g = t_j$.
It follows that $H + t_i + g = H + t_j$, or equivalently, $H + t_i + H + g = H + t_j$. However, this contradicts Lemma \ref{GX-aperiodic}, because it implies that $H + g$ is a period of $X/H$ in $\mathbb{Z}_n/H$.  Thus, Case 2 reduces to a situation similar to Case 1, and it suffices to consider Case 1 to complete the proof.

Assume that $\Gamma$ admits a total perfect code. From Case 1, it follows that $p=2$, and therefore $n/|H|=2k$. Moreover, from \eqref{Total-Z-Z/H}, $\mathrm{CS}(\mathbb{Z}_n/H, S/H)$ also admits a total perfect code, say $C/H$. Moreover, it can be seen that $\mathbb{Z}_n/H\simeq \mathbb{Z}_{2k}$ via the group
isomorphism
$$\varphi : \mathbb{Z}_n/H\rightarrow \mathbb{Z}_{2k}, ~~~\varphi(H+x)=x, ~\text{where}~0\leqslant x<2k. $$
Therefore, by Lemma \ref{IsomorphicGroups}, $\mathrm{CS}(\mathbb{Z}_{2k},T)$, where $T$ is as \eqref{SetT}, admits $C$ as a total perfect code.
 Hence, it follows from Case 1 that $t_i\not\equiv t_j~ (\mathnormal{mod}\ k)$, for distinct $i, j\in \{1, 2, \cdots, k\}$, as required. 

Let $p=2$, and suppose that $t_i\not\equiv t_j~ (\mathnormal{mod}\ k)$, for distinct $i, j\in \{1, 2, \cdots, k\}$. From Case 1, $\mathrm{CS}(\mathbb{Z}_{2k},T)$  admits a total perfect code. By the preceding discussion, it follows that $\mathrm{CS}(\mathbb{Z}_n/H, S/H)$ also admits a total perfect code. Then, by \eqref{Total-Z-Z/H}, $\Gamma$ likewise admits  a total perfect code.

Moreover, from Case 1, any total perfect code of $\mathrm{CS}(\mathbb{Z}_{2k},T)$ is a coset of the subgroup $k\mathbb{Z}_{2k}=\{0,k\}$ in $\mathbb{Z}_{2k}$, and hence, is of the form $\{i, i+k\}$, where $0\leqslant i\leqslant k-1$.  Again by the discussion above, the total perfect codes of $\mathrm{CS}(\mathbb{Z}_n/H, S/H)$ are of the form $H+\{i, i+k\}$, where $0\leqslant i\leqslant k-1$. Therefore,  for any $H+x \in \mathbb{Z}_n/H$ , there exists a unique $H+c \in H+\{i, i+k\}$ such that $H+x+c\in S/H$.

Since $H$ is the stabilizer of $S$ in $\bz_n$, $h+x+c \in S$, for all $h
\in H$. Hence, the total perfect codes of $\Gamma$ are of the general form
$\{h_1+i, h_2+i+k\}$, where $h_1,h_2\in H$ and $0\leqslant i\leqslant k$.
\end{proof}
Before continuing with this section, we review the definition of the polynomial $f_A(x)$ that is associated with the non-empty set $A$ as follows
\begin{equation}\label{f_A}
f_A(x)=\sum_{a\in A}x^a.
\end{equation}
 This polynomial has been defined in  \cite{Feng-2017} and the following lemma is the counterpart of \cite[Lemma 2.4]{Feng-2017}  for Cayley sum graphs.
\begin{lemma}\label{LemTotal-q(x)}
For a subset $C\subseteq \bz _n$, we have $C\in \Tau(\bz _n , S)$ if and only if there exists $q(x)\in \bz [x]$  such that
$$
f_{-C}(x)f_S(x)=(x^n-1)q(x)+(x^{n-1}+\cdots +x+1).
$$
\end{lemma}

\begin{proof}
We  know that $C\in \Tau(\bz _n , S)$ if and only if for any $i\in \bz _n$, there exists a unique $c\in C$ such that $i+c\equiv s~ (\mathnormal{mod}\ n)$, for some $s\in S$. So, 
$$\bz _n =\{s+(-c) ~|~ s\in S, c\in C\}.$$
It can be easily seen that 
\begin{align*}
x^n&\equiv 1~ (\mathnormal{mod}\ (x^n-1)),\\
x^{n+1}&\equiv x~( \mathnormal{mod}\ (x^n-1)),\\
x^{n+2}&\equiv x^2~ (\mathnormal{mod}\ (x^n-1)),\\
&\vdots 
\end{align*}
Therefore,  we have
$$f_S(x)f_{-C}(x)=\sum_{s\in S, ~ c\in -C}x^{s+c}\equiv 1+x+\cdots +x^{n-1}~~(\mathnormal{mod}\ (x^n-1)).$$
\end{proof}
The following theorem is similar to \cite[Theorem 1.3] { Feng-2017} but for Cayley sum graphs.
\begin{theorem}\label{p-circulant}
Suppose that $n$ is  a positive integer and $p$ is an odd prime, where $p \mid n$.
The connected Cayley sum graph $\rm{CS}(\bz_n,S)$ of degree $p$ admits a total perfect code if and only if  $s\not\equiv s^{\prime}~ (\mathnormal{mod}\ p)$, for distinct $s, s^{\prime}\in S$.
\end{theorem}
\begin{proof}
The sufficiency follows from Lemma \ref{k|n-distinct}. For the necessity, we use an argument similar to that of \cite[Theorem 1.1] { Feng-2017} but by setting $S=\{s_0, s_1, \dots, s_{p-1}\}$ and using Lemma \ref{LemTotal-q(x)}. Since in the current case  $s_0\neq 0$, we can conclude from the equation (7) in the proof of \cite[Theorem 1.1] { Feng-2017} that all $t_i$ are equal and consider two following cases.
\begin{itemize}
\item[(i)] If all $t_i=0$, then all $s_i$ are multiplications of $p$. So, $\left\langle S\right\rangle$ cannot generate $\bz _n$ and this contradicts the connectivity of $\CS(\bz _n, S)$, see Lemma \ref{Connectivity-CS}.
\item[(ii)] If all $t_i=r$, where $0 <r<p$, then any element of $S$ can be written as $s_i=k_ip+r$, for some integer $k_i$. So, the subgroup generated by $S-S$ will be a subset of $p\bz_n\simeq \bz_{n/p}$, and hence 
$$|\bz _n : \left\langle S-S\right\rangle |\geqslant \frac{n}{n/p}=p.$$
Therefore $|\bz _n : \left\langle S-S\right\rangle |$ is always greater than 2 because $p$ is an odd prime which is a contradiction with Lemma \ref{Connectivity-CS} and the connectivity of $\CS(\bz_n,S)$.
\end{itemize}
\end{proof}
The following theorem is a generalization of the above theorem which is proved by an argument similar to that of 
 \cite[Theorem 1.2]{ Feng-2017} by setting $S=\{s_0, s_1, \dots, s_{p^l-1}\}$ and using Lemma \ref{LemTotal-q(x)}. Thus, we omit its proof.
\begin{theorem}\label{pl-circulant}
Suppose that $n$ and $l$ are positive integers and $p$ is a prime with  $p^l  \mid n$, and  $p^{l+1}\nmid n$.
The connected Cayley sum graph $\rm{CS}(\bz_n,S)$ of degree $p^l$ has a total perfect code if and only if $s\not\equiv s^{\prime}~ (\mathnormal{mod}\ p^l)$, for distinct $s, s^{\prime}\in S$.
\end{theorem}
\begin{theorem}\label{Periodic-pq}
Suppose that $n\geqslant 6$ is an integer and $p$ and $q$ are primes where $pq\mid n$ and $\left(pq, n/(pq)\right)=1$. Let $S$ is a square-free and periodic subset of $\bz_n$ with $|S|=pq$. Then the connected Cayley sum graph $\CS(\bz_n,S)$ admits a total perfect code if and only if  $s\not\equiv s^{\prime}~ (\mathnormal{mod}\ pq)$, for distinct $s, s^{\prime}\in S$.
 \end{theorem}
 \begin{proof}
 The sufficiency follows from Lemma \ref{k|n-distinct}. 
 
For the necessity, suppose that $n=kpq$, where $(k,pq)=1$, $C\in\Tau(\bz_n,S)$ and  $H$ is the subgroup of periods of $S$ in $\bz_n$, where $|H|>1$.  Since $|H| \mid |S|$, $|H|\in \{p,q,pq\}$. So, we continue the proof in the following two cases. 

\textsf{Case 1.} $|H|=pq=|S|$. Since $\left(pq, n/(pq)\right)=1$, the proof is completed from Lemma \ref{|H|=|X|}. 

\textsf{Case 2.} $|H|\in \{p, q\}$. without loss of generality, we suppose that $|H|=p$. So, $H=(kq)\bz_n$ and $|S/H|=q$. 
By applying an argument similar to that of Case 2 in Theorem \ref{S/H}, we can see that $S\cap H=\emptyset$ and we set $S/H=\{H+t_1, \dots, H+t_q\}$, where $0\leqslant t_i \leqslant kq-1$. Also, from \eqref{Total-Z-Z/H}, $\CS(\bz _n/H, S/H)$ admits a total perfect code. Since,
we  have $\bz _n/H\simeq \bz _{kq}$, $\CS(\bz _{kq},T)$ admits a total perfect code, where $T=\{t_1, t_2, \dots , t_q\}$. By Theorem \ref{p-circulant}, it follows that
$t\not\equiv t^{\prime}~ (\mathnormal{mod}\ |T|=q)$, for distinct $t, t^{\prime}\in T$.

For distinct $s, s^{\prime}\in S$, we have $s\in H+t$ and $s^{\prime}\in H+t^{\prime}$, for some $t, t^{\prime}\in T$. If $t=t^{\prime}$, then $s- s^{\prime}=ckq$ and $c$ is an integer with $0<|c|\leqslant p-1$. We know that $(k,pq)=1$ and $p$ is not a divisor of $ck$, and so, $s\not\equiv s^{\prime}~ (\mathnormal{mod}\ pq)$. If 
$t\neq t^{\prime}$, then $s- s^{\prime}=ckq+t- t^{\prime}$, where $c$ is an integer such that $|c|\leqslant p-1$. Since $q$ does not divide   
$t- t^{\prime}$, $q$ is not also a divisor of $s- s^{\prime}$. Hence, $s \not\equiv s^{\prime}~ (\mathnormal{mod}\ pq)$. 
 \end{proof}
 Before discussing the next theorem, we need to present some preliminary lemmas. The following lemma is concluded from \cite[Lemma 2.2]{Wang-2024}.
 \begin{lemma}\label{Cs-Partition}
 Let $G$ be an abelian group and $S$ be a  square-free subset of $G$.
 If $C\subseteq G$ is a total perfect code of $\CS(G,S)$, then $\{-C+s: s\in S\}$ is a partition of $G$.
 \end{lemma}
 It is proved that for a Cayley graph with vertices $G$ and connection set  $S$ with a total perfect code $C$ that $C+s$ is also a total perfect code, where $s\in S$, see \cite[Lemma 3.1]{Zhou-2016}. The same statement for subgroup total perfect codes in Cayley sum graphs  is proved \cite[Corollary 2.6]{Wang-2024}. We will prove an equivalent statement for $C\in\Tau(G,S)$, when $C\subseteq G$.
 \begin{lemma}\label{C+s}
 Suppose that $G$ is an abelian group and $S$ is a  square-free subset of $G$.
 If $C\subseteq G$ is a total perfect code of $\CS(G,S)$, then $-C+s$ is also a total perfect code of $\CS(G,S)$, for each $s\in S$. 
 \end{lemma}
 \begin{proof}
 Set $S=\{s_1, \dots, s_k\}$. From Lemma \ref{Cs-Partition}, we know that $\{-C+s_1, \dots, -C+s_k\}$ is a partition of $G$. We prove that each vertex $g\in G$ has a unique neighbor in $-C+s_i$, where $1\leqslant i\leqslant k$. Suppose to the contrary. Let $-c_1+s_i$ and $-c_2+s_i$ be the neighbors of $g$ in $-C+s_i$, where $c_1\neq c_2$. So, we have $g -c_1+s_i =s_j$ and $g -c_2+s_i =s_k$, for some $s_j,s_k \in S$.  We have $g+s_i=c_1+s_j=c_2+s_k$, and so,  $-c_1+s_k=-c_2+s_j$ which means that $(-C+s_k )\cap (-C+s_j)\neq\emptyset$ and this is a contradiction with Lemma \ref{Cs-Partition}.
 \end{proof}
 \begin{remark}\label{0InC}
Let $C\in \Tau(G,S)$. Without loss of generality, we can suppose that $0\in C$ because otherwise if $0\notin C$, then there exists a unique $c\in C$ such that $0+c=s$, for some $s\in S$. So, $0\in -C+s$ which is a total perfect code of $\CS(G,S)$ by Lemma \ref{C+s}. 
\end{remark}
\begin{lemma}\label{Lem-pq}
Let $n\in \{pqr, pqrs: p,q,r,s ~\text{are distinct primes}\}$. Suppose that $S\subseteq \bz_n$ is square-free and aperiodic with $|S|=pq$. Moreover, suppose that $|S| \mid n$ and $(|S|, n/|S|)=1$. The connected Cayley sum graph $\CS(\bz_n,S)$ admits a total perfect code if and only if  $s\not\equiv s^{\prime}~ (\mathnormal{mod}\ pq)$, for distinct $s, s^{\prime}\in S$.
\end{lemma}
\begin{proof}
The sufficiency follows from Lemma \ref{k|n-distinct}.

To prove the necessity, suppose that $n=kpq$ with $k\geqslant 2$ and $(pq,k)=1$, and $C\in \Tau (\bz_n,S)$ with $0\in C$, see Remark \ref{0InC}. From Remark \ref{trans-fac-total}, we have $\bz _n=(-C)\oplus S$, and so from Lemma \ref{factor-conditions}, we have $|C|=n/(pq)=k$. Since $\bz_n$ is a good abelian group, see \eqref{GoodAbeli}, and $S$ is aperiodic, $-C$ must be periodic. Let $H$ be the subgroup of priods of $-C$ in $\bz _n$. We have $|H|\geqslant 2$ and $|H|\mid |C|$. Since $0\in (-C)$, $H\subseteq (-C)$.
By Lemma \ref{factor-GX}, $\bz_n/H=(S/H)\oplus (-C/H)$ and $H\cap (S-S)=\{0\}$. It follows that $H+s\neq H+s^{\prime}$, for distinct $s, s^{\prime}\in S$, because otherwise we would have $(s-s^{\prime})\in H\cap (S-S)$ which is a contradiction. Thus, $|S/H|=|S|=pq$ and we set $S/H=\{H+t_1, \dots, H+t_{pq}\}$ with $0\leqslant t_i <l$ for each $i$, where $l= n/|H|$.  Set 
\begin{equation}\label{SetT2}
T=\{t_1, t_2, \dots, t_{pq}\}.
\end{equation}
We now prove that $S\cap H=\emptyset$ or equivalently $H\notin S/H$. Suppose to the contrary and let $h\in H\cap S$. 
Since $C\in \Tau(\bz_n, S)$, there exists a unique $c\in C$ such that $h+c\in S$. So, $(-c)+(c+h)$ and $0+h$ are two different expressions for $h$, contradicting the fact that $\bz _n=(-C)\oplus S$. Since we have $\bz_n/H=(S/H)\oplus (-C/H)$, $C/H \in \Tau(\bz_n/H, S/H)$  by Remark \ref{trans-fac-total}.
Note that $\bz_n/H\simeq \bz_l$ via the group isomorphism $\varphi(H+x)=x$, and so by Lemma \ref{IsomorphicGroups}, $\CS(\bz_l, T)$, where $T$ is as in \eqref{SetT2}, also admits a total perfect code. Moreover, the Cayley sum graphs $\CS(\bz_n/H, S/H)$ and $\CS(\bz_l, T)$ are connected  because of the connectivity of  $\CS(\bz_n, S)$. By assumption, we know that $n\in \{pqr, pqrs: p,q,r,s ~\text{are distinct primes and}~ k\geqslant 1\}$, and so, $l$ belongs to $\{pq, pqr, pqs\}$ and $\bz _n/H$ is a good abelian group, see \eqref{GoodAbeli}. Since  $\bz _n/H=(-C/H)\oplus (S/H)$, we conclude from Lemma \ref{GX-aperiodic} that $-C/H$ is aperiodic in $\bz_n/H$, and hence, $S/H$ must be periodic. Suppose that $H+b\in \bz _n/H$, where $b\notin H$, is a period of $S/H$ or equivalently, $S/H + H+b=S/H$. Therefore, we have $T+b=T$ and $b$ is a period of $T$ in $\bz_l$. In particular, $b$ is not the identity element of $\bz_l$ because $\varphi$ is an isomorphism. Since $(k, pq)=1$ and $l=\left(k/|H|\right)(pq)$, we have $(l/(pq),pq)=(k/|H|,pq)=1$, and hence, $\CS(\bz_l,T)$
satisfies the assumptions of Theorem \ref{Periodic-pq}.  Therefore, we conclude that 
\begin{equation}\label{ttPrime}
t\not\equiv t^{\prime}~ (\mathnormal{mod}\ pq),~ \text{for distinct}~ t, t^{\prime}\in T.
\end{equation}
We now suppose that $s, s^{\prime}$ are distinct elements of $S$. We have $s \in H+t$ and $s^{\prime} \in H+t^{\prime}$, for some $t, t^{\prime}\in T$. If $t=t^{\prime}$, then $s-s^{\prime}\in H$ which contradicts $H\cap (S-S^{\prime})=\{0\}$. So, we have $t\neq t^{\prime}$, and consequently, $s-s^{\prime}=cl+t-t^{\prime}$ where $c$ is an integer with $|c|<|H|$. From \eqref{ttPrime}, $pq$ does not divide $t-t^{\prime}$ and we also know that $pq$ is a divisor of $l$. So, we obtain $s\not\equiv s^{\prime}~ (\mathnormal{mod}\ pq)$.
\end{proof}
\begin{theorem}\label{n-in-N}
Suppose that $n\in N$, where $N$ is the set in \eqref{GoodAbeli}, and $S$ is a square-free subset of $\bz_n$ such that $|S|\mid n$ and $\left(|S|, n/|S| \right)=1$.
Then the connected Cayley sum graph $\CS(\bz_n, S)$ admits a total perfect code if and only if $s\not\equiv s^{\prime}~ (\mathnormal{mod}\ |S|)$ for distinct $s, s^{\prime}\in S$.
\end{theorem}
\begin{proof}
The sufficiency is proved by Lemma \ref{k|n-distinct}.

For the nessecity, let $\CS(\bz_n, S)$ with $(|S|, n/|S|)=1$ which admits a total perfect code $C$. From Remark \ref{trans-fac-total}, we have $\bz_n=(-C)\oplus S$. We want to prove that 
\begin{equation}\label{ssPrime}
s\not\equiv s^{\prime}~ (\mathnormal{mod}\ |S|),~ \text{for distinct}~ s, s^{\prime}\in S.
\end{equation}
We continue the proof according to different values that $n$ can take as follows.

\textsf{Case 1.} $n=p^k$, where $k$ is a positive integer. In this case, we must have $|S|=1$, because otherwise the condition $(|S|, n/|S|)=1$
is not satisfied. So,  $\CS(\bz_n, S)$ is 1-regular and is not connected unless it is $K_2$.

\textsf{Case 2.} $n=p^kq$, where $p$ and $q$ are distinct primes and $k\geqslant 1$. So, $|S|\in \{q, p^k\}$. By applying Theorem \ref{pl-circulant}, the statement \eqref{ssPrime} is proved.

\textsf{Case 3.} $n=p^2q^2$, where $p$ and $q$ are distinct primes. In this case, we have $|S|\in \{p^2, q^2\}$, and so, \eqref{ssPrime} is obtained from Theorem \ref{pl-circulant}.

\textsf{Case 4.} $n=pqr$, where $p$, $q$ and $r$ are distinct primes. So, $|S|\in \{p, q, r, pq, pr, qr\}$ and \eqref{ssPrime} follows from Theorem \ref{p-circulant} and Lemma \ref{Lem-pq}.

\textsf{Case 5.} $n=p^2qr$, where $p$, $q$ and $r$ are distinct primes.
 If $|S|\in \{r, q\}$, then the proof is completed by Theorem \ref{p-circulant}. If $|S|=p^2$ or $|S|=qr$, then \eqref{ssPrime} follows from Theorem \ref{pl-circulant} and Lemma \ref{Lem-pq}, respectively.
 We now suppose that $|S|\in \{p^2q, p^2 r\}$. we will have the following two subcases.
 \begin{itemize}
 \item[(5.1)]
  If $S$ is aperiodic, then
 we obtain from Lemma \ref{factor-conditions} that $n=|C||S|=r|S|$ or $q|S|$. By applying Theorem \ref{S/H}, \eqref{ssPrime} is concluded.
 \item[(5.2)]
 Let $S$ is periodic and $H$ be the subgroup of periods of $S$ in $\bz_n$. We claim that $H\notin S/H$ or $H\cap S=\emptyset$. Suppose to the contrary that $h\in S\cap H$.  Then $2h\in S+h=S$ which contradicts the fact that $S$ is square-free. If $|H|=|S|$, then \eqref{ssPrime} is obtained from Lemma \ref{|H|=|X|}. We now suppose that $|H|\neq |S|$. Since $|H|\mid |S|$, we have $|H|\in\{p, p^2, q, pq, r, pr\}$. Set $k=|S|/|H|$, $l=n/|H|$, $S/H=\{H+t_1, H+t_2, \dots, H+t_k\}$, and $T=\{t_1, t_2, \dots, t_k\}$, where $0<t_i<l$. We can see that $\bz_n/H\simeq \bz _l$ by the group isomorphism $\varphi (H+x)=x$. So, $\CS(\bz_l , T)$ admits a total perfect code by 
 \eqref{Total-Z-Z/H}. Since $k=|T|\in \{p, p^2, q, r, pq, pr\}$,  we conclude by Theorem \ref{pl-circulant}, Theorem \ref{Periodic-pq}, and Lemma \ref{Lem-pq} that $t\not\equiv t^{\prime}~ (\mathnormal{mod}\ k)$ for distinct $t, t^{\prime}\in T$. To obtain \eqref{ssPrime}, let $s\in H+t$ and $s^{\prime}\in H+t^{\prime}$, for some $t, t^{\prime}\in T$.
 If $t=t^{\prime}$, then $s-s^{\prime}=cl$, where $0<|c|<|H|$. Note that $l=kn/|S|$ and $(|S|, n/|S|)=1$, so, $|S|$ does not divide $cl$ and 
$s \not\equiv s^{\prime}~ (\mathnormal{mod}\ |S|)$. If $t\neq t^{\prime}$, then $s-s^{\prime}=cl +t -t^{\prime}$, where $c$ is an integer with $|c|<|H|$. Since $k$ does not divide $ t-t^{\prime}$ and $l=kn/|S|$, we have $s \not\equiv s^{\prime}~ (\mathnormal{mod}\ k)$. So, we obtain
 $s \not\equiv s^{\prime}~ (\mathnormal{mod}\ |S|)$ because $|S|=k|H|$.
 \end{itemize}
 \textsf{Case 6.} $n=pqrs$, where $p$, $q$, $r$ and $s$ are distinct primes. If $|S|\in\{p, q, r, s\}$,  then \eqref{ssPrime} is obtained from Lemma \ref{p-circulant}. If $|S|$ is the multiplication of two of $p$, $q$, $r$ and $s$, then \eqref{ssPrime} follows from Lemma \ref{Lem-pq}. We now suppose  that $|S|$ is as the product of three of $p$, $q$, $r$ and $s$. Without loss of generality, let $|S|=qrs$. So,  we have $|C|=n/|S|=p$ by Lemma \ref{factor-conditions}. We now consider two subcases as follows.
 \begin{itemize}
 \item[(6.1)] If $S$ is aperiodic, then \eqref{ssPrime} is obtained from Theorem \ref{S/H}.
 \item[(6.2)] Let $S$ be periodic and $H$ be the subgroup of periods of $S$ in $\bz_n$ with $|H|>1$. By an argument similar to that of subcase 5.2, we have $H\notin S/H$. If $|S|=|H|$, then \eqref{ssPrime} follows from Lemma \ref{|H|=|X|}. We now suppose that $|S|\neq |H|$. We know that $|H| \mid |S|$. Set $k=|S|/|H|$,  $l=n/|H|$,  $S/H=\{H+t_1, H+t_2, \dots, H+t_k\}$ and $T=\{t_1, t_2, \dots, t_k\}$, where $0< t_i < l$ for each $i$. Since $\bz_n=(-C)\oplus S$, we obtain from Lemma \ref{factor-GX} that $\bz_n/H=\left(-C/H\right)\oplus S/H$,  and so by Remark \ref{trans-fac-total}, we have $C/H\in \Tau (\bz_n/H, S/H)$. Since $\bz_n/H\simeq \bz_l$ by the group isomorphism $\varphi: \bz_n/H\rightarrow \bz_l$ defined by $\varphi(H+x)=x$, we have from Lemma \ref{IsomorphicGroups} that $\CS(\bz_l,T)$ also admits a total perfect code. Note that $|T|=k\in \{q, r, s, qr, qs,rs\}$ and by Theorem \ref{p-circulant} and Lemma \ref{Lem-pq}, we obtain $t \not\equiv t^{\prime}~ (\mathnormal{mod}\ k)$ for distinct $t, t^{\prime}\in T$. The rest of the proof is completed by the same way as subcase 5.2.
 
 \end{itemize}
\end{proof}
\section{Cayley Sum Graphs over Direct Products of Cyclic Groups}\label{Direct-Sec}
In this section,  we aim to generalize the results of the previous section to Cayley sum graphs of the direct product of cyclic groups. Let
\begin{equation}\label{DirectProduct}
G=\bz_{n_1}\times \bz_{n_1}\times \dots \times \bz_{n_d},
\end{equation}
where $n_1, n_2, \dots, n_d\geqslant 2$.  So, $G$ is of the order $n=n_1 n_2\dots n_d$.
Moreover, suppose that $A$ is a non-empty set of $G$ . A polynomial is associated with $A$, similar to \eqref{f_A}, as follows \cite{Cameron-2025}
$$f_A(x_1,\dots, x_d)=\sum _{(a_1,\dots,a_d)\in A}x_1^{a_1}\dots x_d^{a_d},$$
where for any  $(a_1,\dots,a_d)\in A$, $a_i\in \bz_{n_i}$ is an integer between $0$ and $n_i-1$. 
\begin{lemma}
Suppose that $G$ is as defined in \eqref{DirectProduct} and $S\subseteq G$. The Cayley sum graph $\CS(G,S)$ admits $C\subseteq G$ as a total perfect code if and only if  for each pair $(I,g)$, where $\emptyset \neq I \subseteq G$ and $g=(g_1, \dots, g_d)\in G$, there exists a polynomial $q_I^{(g)}(x_1,\dots,x_d)\in \bz [x_1,\dots, x_d]$ divided by $\left( \prod_{i=1}^d x_i^{g_i}\right)\left( \prod_{i\in I} (x_i^{n_i}-1)\right)$ such that
$$f_{-C}(x_1, \dots, x_d)f_{S}(x_1, \dots, x_d)=\sum_{\emptyset \neq I\subseteq \{1,\dots,d\}}\sum_{g\in G}q_I^{(g)}(x_1,\dots,x_d)+
\prod _{i=1}^d\left( \sum_{j=0}^{n_i-1}x_i^j\right).$$
\end{lemma}
\begin{proof}
Let $C\subseteq G$ be a total perfect code of $\CS(G,S)$. By Remark \ref{trans-fac-total}, we have $G=(-C)\oplus S$. So, every $g\in G$ can be uniquely written as 
$$(g_1, \dots, g_d)=(c_1+ s_1~(\mathnormal{mod}\ n_1),\dots, c_d+ s_d~(\mathnormal{mod} \ n_d) ),$$
where $(c_1,\dots , c_d)\in -C$ and $(s_1,\dots , s_d)\in S$. The remaining part of the necessity proof follows the same argument as in \cite[Lemma 4.1]{Cameron-2025}. The sufficiency part is also similar to the proof of \cite[Lemma 4.1]{Cameron-2025}, with the slight difference that we ultimately conclude that $G=(-C)\oplus S$, and thus from Remark \ref{trans-fac-total}, $C$ is a total perfect code of $G$.
\end{proof}
\begin{lemma}\label{SufficiencyLem}
Suppose that $G$ is as \eqref{DirectProduct} and $S\subseteq G$, where $|S|=m_1\dots m_d$ with $m_i\geqslant 1$ and $m_i\mid n_i$ for each $i\in \{1,\dots ,d\}$. If for each pair of distinct elements $(s_1,\dots, s_d), (s_1^{\prime},\dots, s_d^{\prime})\in S$ there exists at least one $j\in \{1,\dots ,d\}$ with $s_j\not\equiv s_j^{\prime}~(\mathnormal{mod}\ m_j)$, then the connected Cayley sum graph $\CS(G,S)$ admits a total perfect code.
\end{lemma}
\begin{proof}
Set $$C=(m_1\bz _{n_1})\times \dots \times (m_d\bz _{n_d}).$$
By an argument similar to that of \cite[Lemma 4.2]{Cameron-2025}, we have $|G|=|-C||S|$ and 
$$\left((-C)-(-C)\right)\cap (S-S)=\{(0,\dots , 0)\}.$$
So from Lemma \ref{factor-conditions}, we obtain $G=(-C)\oplus S$. Therefore, we conclude that $C$ is a total perfect code of $\CS(G,S)$ by Remark \ref{trans-fac-total}.
\end{proof}
We see in the following example that the sufficient condition in Lemma \ref{SufficiencyLem} is not always necessary.
\begin{example}
Let $G=\bz _4 \times \bz _4$ and  $S=\{ (0,1), (1,1), (1,3), (3,2)\}$.  We can check that $\left\langle S\right\rangle =G$ and $\left\langle S-S\right\rangle =G$, and so from Lemma \ref{Connectivity-CS},  the Cayley sum graph $\CS(G, S)$ is connected which admits $C=\{(0,1), (1,2),(2,3), (3,0)\}$ as a total perfect code, see Figure \ref{Fig1}. We can factorize $|S|$ in two ways as $m_1m_2$ such that $m_1 \mid n_1$ and $m_2 \mid n_2$, namely $(m_1,m_2)=(1,4)$ and $(m_1,m_2)=(4,1)$. In the first case, we have
$(1-0 ~(\mathnormal{mod}\ 1), 1-1~ (\mathnormal{mod}\ 4))=(0,0)$ and in the second case, we obtain $(1-1 ~(\mathnormal{mod}\ 4), 3-1 ~ (\mathnormal{mod}\ 1))=(0,0)$.
\end{example}
\begin{figure}[h] 
\centering
\begin{tikzpicture}[scale=.5, transform shape]

\node [draw, shape=rectangle, minimum width=1cm, minimum height=1cm,fill=green] (v1) at (7.00,-0.00) {$(2,3)$};
\node [draw, shape=rectangle, minimum width=1cm, minimum height=1cm,fill=orange] (v2) at (-0.00,-7.00) {$(0,1)$};
\node [draw, shape=rectangle, minimum width=1cm, minimum height=1cm,fill=orange] (v3) at (-0.00,7.00) {$(1,2)$};
\node [draw, shape=rectangle, minimum width=1cm, minimum height=1cm,fill=green] (v4) at (-7.00,-0.00) {$(3,0)$};

\node [draw, shape=circle,fill=orange] (v5) at (3.00,-6.50) {$(0,0)$};
\node [draw, shape=circle,fill=orange] (v6) at (5.00,-4.90) {$(3,1)$};
\node [draw, shape=circle,fill=green] (v7) at (6.80,-2.80) {$(3,2)$};
\node [draw, shape=circle,fill=green] (v8) at (6.50,2.50) {$(2,2)$};
\node [draw, shape=circle,fill=green] (v9) at (5.30,4.50) {$(1,3)$};
\node [draw, shape=circle,fill=orange] (v10) at (3.20,6.30) {$(3,3)$};
\node [draw, shape=circle,fill=orange] (v11) at (-3.00,6.30) {$(2,0)$};
\node [draw, shape=circle,fill=orange] (v12) at (-5.10,4.70) {$(0,3)$};
\node [draw, shape=circle,fill=green] (v13) at (-6.50,2.50) {$(2,1)$};
\node [draw, shape=circle,fill=green] (v14) at (-6.50,-2.60) {$(1,1)$};
\node [draw, shape=circle,fill=green] (v15) at (-5.00,-4.80) {$(0,2)$};
\node [draw, shape=circle,fill=orange] (v16) at (-2.80,-6.30) {$(1,0)$};

\draw[line width=1pt,color=orange, bend right=20] (0.56,-7.00) to (2.34,-6.50);
\draw[line width=1pt,color=orange, bend right=60] (0.55,-7.50) to (5.00,-4.90);
\draw[line width=1pt,color=orange, bend left=20] (-0.56,-7.00) to (-2.18,-6.40);
\draw[line width=1pt,color=orange, bend left=20] (-0.56,-7.00) to (-2.18,-6.40);
\draw[line width=1pt,color=orange] (-0.07,-6.50) -- (0.07,6.50);
\draw[line width=1pt,color=gray, bend right=60] (3.60,-6.30) to (6.60,-3.45);
\draw[line width=1pt,color=gray] (3.20,-5.90) -- (5.00,3.90);
\draw[line width=1pt,color=gray] (2.60,-6.00) -- (-5.85,-2.60);
\draw[line width=1pt,color=gray] (5.00,-4.20) -- (6.20,1.90);
\draw[line width=1pt,color=gray] (4.45,-4.50) -- (-2.20,-6.00);
\draw[line width=1pt,color=gray] (4.70,-4.30) -- (-2.90,5.60);
\draw[line width=1pt,color=gray] (-2.70,-5.60)-- (6.00,2.00);
\draw[line width=1pt,color=gray] (-3.10,-5.70)-- (-5.10,4.00);
\draw[line width=1pt,color=green, bend right=20] (6.80,-2.10) to (7.00,-0.50);
\draw[line width=1pt,color=green, bend right=20] (7.00,0.50) to (6.50,1.80);
\draw[line width=1pt,color=green, bend right=80] (7.60,0.40) to (6.00,4.50);
\draw[line width=1pt,color=green] (6.40,0.10) -- (-6.40,-0.00);
\draw[line width=1pt,color=gray] (6.40,-2.20) -- (5.30,3.80);
\draw[line width=1pt,color=gray] (6.40,-2.20) -- (5.30,3.80);
\draw[line width=1pt,color=gray] (6.20,-2.40) -- (-5.80,2.50);
\draw[line width=1pt,color=gray, bend right=80] (6.50,3.20) to (3.90,6.30);
\draw[line width=1pt,color=gray] (4.60,4.40)--(-4.60,-4.20);
\draw[line width=1pt,color=orange, bend right=20] (2.55,6.60) to (0.60,7.00);
\draw[line width=1pt,color=orange, bend right=20] (-0.60,7.00) to (-2.34,6.50);
\draw[line width=1pt,color=orange, bend right=80] (-0.20,7.50) to (-5.00,4.30);
\draw[line width=1pt,color=gray, bend right=80] (3.10,7.00) to (-2.70,6.90);
\draw[line width=1pt,color=gray] (2.55,6.10) -- (-4.40,4.80);
\draw[line width=1pt,color=gray, bend right=80] (-3.40,6.82) to (-6.50,3.20);
\draw[line width=1pt,color=gray] (-5.40,4.10) -- (-5.00,-4.10);
\draw[line width=1pt,color=green, bend right=20] (-6.80,1.90) to (-7.00,0.54);
\draw[line width=1pt,color=green, bend right=20] (-7.00,-0.53) to (-6.90,-2.00);
\draw[line width=1pt,color=green, bend right=80] (-7.60,-0.00) to (-5.70,-4.80);
\draw[line width=1pt,color=gray, bend right=80] (-7.20,2.30) to (-7.10,-2.30);
\draw[line width=1pt,color=gray, bend right=20] (-6.50,-3.30) to (-5.50,-4.30);

\end{tikzpicture}
\caption{ $\mathrm{CS}(\bz_4\times \bz_4, \{ (0,1), (1,1), (1,3), (3,2)\})$}\label{Fig1}
\end{figure}

\begin{theorem}\label{Product-P}
Let $G$ be as \eqref{DirectProduct} and $S\subseteq G$, where $|S|$ is an odd prime and divides exactly one of $n_1, \dots, n_d$, say $n_t$.
The connected Cayley sum graph $\CS(G,S)$ admits a total perfect code if and only if $s_t\not\equiv s_t^{\prime}~(\mathnormal{mod}\ p)$, for each pair of distinct elements $(s_1,\dots, s_d), (s_1^{\prime},\dots, s_d^{\prime})\in S$. 
\end{theorem}
\begin{proof}
The sufficiency follows from Lemma \ref{SufficiencyLem}. The necessity part is proved in the same way as \cite[Theorem 4.3]{Cameron-2025} but with $S_0$ replaced by $S$ and differs only in Claim 2, where  we obtain $p \mid \sum_{i=1}^px^{r_i}$ and we need to consider the following items. 
\begin{itemize}
\item For all $1\leqslant i\leqslant p$, $r_i=0$. In this case, we have $s_1^{(i)}=k_ip$, for some integer $k_i$. So, all $s_1^{(i)}$ are multiplications of $p$ and $S$ cannot generate $G$ which contradicts the connectivity of $\CS(G,S)$, see Lemma \ref{Connectivity-CS}.
\item For all $1\leqslant i\leqslant p$, $r_i=r\neq 0$. Hence, we obtain $s_1^{(i)}=k_ip+r$ for some integer $k_i$. Therefore, $\left\langle S-S\right\rangle$ is the set of elements whose first coordinate is a multiplication of $p$ in $\bz_{n_1}$. So, the maximum value that $|\left\langle S-S\right\rangle|$ can take is $|G|/p$ and we obtain
$$|G: \left\langle S-S\right\rangle |\geqslant p.$$
Since $p$ is an odd prime, we have $|G: \left\langle S-S\right\rangle |>2$ which contradicts Lemma \ref{Connectivity-CS}.
\end{itemize}
\end{proof}
The following example shows that  the condition that $|S|$ divides only one $n_i$, $0\leqslant i \leqslant d$, in Theorem \ref{Product-P} cannot be omitted.
\begin{example}
Suppose that $G=\bz _3 \times \bz _6$  and $S=\{(0,3), (0,1),(1,1)\}$ is a subset of $G$. It is easy to check that $\left\langle S \right\rangle =G$ and $|G:\left\langle S-S \right\rangle|\leqslant 2$. So, the Cayley sum graph $\CS (G,S)$ is connected by Lemma \ref{Connectivity-CS}. Also, $\CS (G,S)$ admits $C=\{(0,0),(0,1), (1,2), (1,3), (2,4), (2,5)\}$ as a total perfect code, see Figure \ref{Fig2}. We have $|S|\mid n_1$ and $|S|\mid n_2$. We can see that $s_1^{(1)}\equiv s_1^{(2)} ~(\mathnormal{mod}\ |S|)$ and $s_2^{(2)}\equiv s_2^{(3)} ~(\mathnormal{mod}\ |S|)$, for $S^{(1)}=(s_1^{(1)},s_2^{(1)})=(0,3)$,   $S^{(2)}=(s_1^{(2)},s_2^{(2)})=(0,1)$ and $S^{(3)}=(s_1^{(3)},s_2^{(3)})=(1,1)$.
\end{example}
\begin{figure}[h]
\centering
\begin{tikzpicture}[scale=.5, transform shape]

\node [draw, shape=rectangle, minimum width=1cm, minimum height=1cm,fill=green] at  (-6.0,0) {$(2,4)$};
\node [draw, shape=rectangle, minimum width=1cm, minimum height=1cm,fill=orange] at  (-3.49,5.09) {$(0,0)$};
\node [draw, shape=rectangle, minimum width=1cm, minimum height=1cm,fill=orange] at  (3.49,5.09) {$(0,1)$};
\node [draw, shape=circle,fill=green] at  (-5.79,1.74) {$(1,5)$};
\node [draw, shape=circle,fill=orange] at  (-5.09,3.42) {$(1,1)$};
\node [draw, shape=circle,fill=pink] at  (5.79,1.74) {$(2,1)$};
\node [draw, shape=circle,fill=orange] at  (5.09,3.42) {$(0,2)$};
\node [draw, shape=rectangle, minimum width=1cm, minimum height=1cm,fill=pink] at  (6.0,0) {$(1,2)$};
\node [draw, shape=circle,fill=orange] at  (-1.19,5.5) {$(0,3)$};
\node [draw, shape=circle,fill=orange] at  (1.19,5.5) {$(1,0)$};
\node [draw, shape=rectangle, minimum width=1cm, minimum height=1cm,fill=pink] at  (3.49,-5.09) {$(2,5)$};
\node [draw, shape=rectangle, minimum width=1cm, minimum height=1cm,fill=green] at  (-3.49,-5.09) {$(1,3)$};
\node [draw, shape=circle,fill=pink] at  (1.19,-5.94) {$(2,2)$};
\node [draw, shape=circle,fill=green] at  (-1.19,-5.94) {$(0,4)$};
\node [draw, shape=circle,fill=pink] at  (5.79,-1.74) {$(0,5)$};
\node [draw, shape=circle,fill=pink] at  (5.09,-3.42) {$(1,4)$};
\node [draw, shape=circle,fill=green] at  (-5.79,-1.74) {$(2,3)$};
\node [draw, shape=circle,fill=green] at  (-5.09,-3.42) {$(2,0)$};

\draw[line width=1pt,color=orange, bend left=70]  (-3.2,5.55) to  (3.49,5.55);
\draw[line width=1pt,color=orange, bend left=20]  (-2.9,5) to   (-1.75,5.85);
\draw[line width=1pt,color=orange, bend right=20]  (-3.4,4.58) to   (-4.54,3.75);
\draw[line width=1pt,color=orange, bend left=20]  (1.7,5.92) to   (2.9,5.2);
\draw[line width=1pt,color=orange, bend left=20]  (3.67,4.6) to   (4.7,3.88);
\draw[line width=1pt,color=gray]  (-4.88,2.75) -- (-4.95,-2.75);
\draw[line width=1pt,color=gray]  (-4.67,2.85) -- (1,-5.3);
\draw[line width=1pt,color=gray]   (-1.19,4.8) -- (-1.19,-5.27);
\draw[line width=1pt,color=gray]   (-0.7,5) -- (4.8,-2.8);
\draw[line width=1pt,color=gray]   (1.1,4.83) -- (-5.4,-1.19);
\draw[line width=1pt,color=gray]   (1.3,4.83) -- (5.3,2.2);
\draw[line width=1pt,color=gray]   (4.46,3.2) -- (-5.1,1.6);
\draw[line width=1pt,color=gray]   (4.7,2.9) -- (5.34,-1.2);
\draw[line width=1pt,color=green, bend right=70]   (-6.59,0) to  (-4.1,-5.09);
\draw[line width=1pt,color=gray]   (-5.5,1.15) --  (0.6,-5.6);
\draw[line width=1pt,color=green, bend right=20]  (-5.86,1.1) to  (-6.0,0.45);
\draw[line width=1pt,color=green, bend right=20]  (-6.0,-0.45) to   (-5.79,-1.3);
\draw[line width=1pt,color=green, bend left=20]  (-4.1,-4.9) to   (-5.23,-4.1);
\draw[line width=1pt,color=green, bend right=20]  (-2.87,-5.3) to    (-1.85,-5.94);
\draw[line width=1pt,color=gray]  (-4.5,-3) --  (5.12,1.7);
\draw[line width=1pt,color=gray]   (-5.1,-1.74) -- (4.4,-3.42);
\draw[line width=1pt,color=gray]   (-0.7,-5.45) -- (5.2,-2);
\draw[line width=1pt,color=pink, bend left=70]   (6.6,-0.07) to  (4.07,-5.09);
\draw[line width=1pt,color=pink, bend left=20]   (6.0,-0.46) to   (5.79,-1.1);
\draw[line width=1pt,color=pink, bend right=20]   (6.0,0.48) to (5.79,1.1);
\draw[line width=1pt,color=pink, bend right=20]  (4.05,-4.81) to (5.09,-4.1);
\draw[line width=1pt,color=pink, bend left=20]  (3.14,-5.6) to (1.87,-5.94);
\end{tikzpicture}
\caption{ $\mathrm{CS}(\bz_3\times \bz_6, \{ (0,3), (0,1), (1,1)\})$}\label{Fig2}
\end{figure}

\begin{theorem}
Let $G$ be as \eqref{DirectProduct} and $S\subseteq G$, where $|S|=p^l$ is a prime power with $p^l \mid n$ and $p^{l+1}\nmid n$. Moreover, suppose that $p$ divides exactly one of $n_1, \dots, n_d$, say $n_t$. The connected Cayley sum graph $\CS(G,S)$ admits a total perfect code if and only if $s_t\not\equiv s_t^{\prime}~(\mathnormal{mod}\ p^l)$, for each pair of distinct elements $(s_1,\dots, s_d), (s_1^{\prime},\dots, s_d^{\prime})\in S$.
\end{theorem}
\begin{proof}
The sufficiency is proved by Lemma \ref{SufficiencyLem}. The necessity part follows from the same way as \cite[Theorem 4.3]{Cameron-2025} but with $S_0$ replaced by $S$.
\end{proof}
  \section*{Acknowledgement}
  The first author was supported by the University of Mazandaran, Grant Number 60673.

\end{document}